\documentclass[letterpaper,11pt]{article}
\usepackage[english]{babel}
\usepackage[ansinew]{inputenc}
\usepackage{fontenc}
\usepackage{amsfonts}
\usepackage{amsmath}
\usepackage{amsthm}
\usepackage{enumerate}
\usepackage{textcomp}
\usepackage{geometry}
\usepackage{mathrsfs}
\usepackage[colorlinks=true,linkcolor=red,citecolor=blue]{hyperref}

\newtheorem{lem}{Lemma}[section]
\newtheorem{cor}{Corollary}[section]
\newtheorem{prp}{Proposition}[section]
\newtheorem{rem}{Remark}[section]

\usepackage{geometry}
\usepackage[colorlinks=true,linkcolor=red,citecolor=blue]{hyperref}

\begin{document}

\title{A simple analysis of a $D/GI/1$ vacation queue with impatient customers}
\author{Assia BOUMAHDAF\footnote{e-mail: assia.boumahdaf@gmail.com}}
\date{}
\maketitle

\begin{abstract}
In this paper, we deal with an $D/GI/1$ vacation system with impatient customers. We give a sufficient condition for the existence of a limit distribution of the waiting time and integral equations are derived in both reneging and balking scenarios. Explicit solutions are given when vacation times are exponentially distributed and service times are either exponentially distributed or deterministic.
\end{abstract}

\maketitle

\section{Introduction}

In this paper, we propose a simple analysis of a $D/GI/1$ vacation system with impatient customers. Such models may be used to describe many data switching systems whose data transmission must be done in a very short time. Such systems may have to support periodic arrival streams and might also execute other tasks that come from other queues. Thus, these secondary tasks may be regarded as server on vacations.

Some investigations have been published regarding the performance of vacation queueing systems with impatient customers. In the literature queueing system may say as limited waiting time or limited sojourn time.
Most of these works focus on $M/G/1 $ queueing models with a general vacation distribution and constant deadline. The study of these queueing systems have been introduced by \cite{vanderDuynSchouten78}. The author derives the joint stationary distribution of the workload and the state of the server (available for service or on vacations). In \cite{TakineHasegawa90}, the authors have considered two $M/G/1$ with balking customers and deterministic deadline on the waiting time and sojourn time. They have obtained integral equations for the steady state probability distribution function of the waiting times and the sojourn times. They expressed these equations in terms of steady state probability distribution function of the $M/G/1$ queue with vacations without deadline. Recently, \cite{KatayamaTsuyoshi2011} has investigated the $ M/G/1 $ queue with multiple and single vacations, sojourn time limits and balking behavior. Explicit solutions for the stationary virtual waiting time distribution are derived under various assumptions on the service time distribution. The same author in \cite{KatayamaTsuyoshi2012}, derives recursive equations in the case of a deterministic service times for the steady-state distributions of the virtual waiting times in a $ M/G/1 $ queue with multiple and single vacations, sojourn time limits and balking behavior. In \cite{AltmanYechiali2006}, the authors have analyzed queueing models in which customers become impatient only when servers are on vacations. They have derived some performance measures for $M/M/1 $, $M/G/1 $ and $ M/G/1 $ for both multiple and single vacations. 

In the case of a single service discipline, we analyse a Lindley-type equation \cite{Lindley52} and the model may be reduce to the $D/GI/1 + D$ queue. A sufficient condition is given for the existence of the stationary waiting time distribution and an integral equation is established for both reneging and balking models in the case of a single service discipline. The stationary probability of rejection is also derived for both models. Using the Laplace transform to solve a differential equation, simple explicit solutions are given when vacation times are exponentially distributed, and when service times distribution are either exponential, or deterministic.

This model was studied by \cite{Ghosal1963} and the more general case was studied, for example by \cite{Daley65}, \cite{Baccelli84}, or \cite{Stanford79}. In \cite{Ghosal1963}, the author derives an integral equation for the stationary waiting time distribution based on the model introduced by \cite{Finch1960} (see also \cite{Ghosal1970}, Chapter 1, equation (1.4)). The same author in \cite{Finch1960} proposes a correction in \cite{Finch1961}. In this paper, we take into account \cite{Finch1961} and  to rewrite the integral equation and solve it in a simple manner under various hypotheses. \vspace{0.5cm}

The paper is organized as follows. In Section 2, we give a description of reneging and balking models. In Section 3, we establish preliminary results. In Section 4, we give a sufficient condition for the stability and derives integral equations for the stationary waiting time distribution and the stationary probability of rejection for reneging and balking models. Section 5 derives an explicit solution of the integral equation in the case of exponentially distributed vacation times and deterministic service times. In Section 6, we focus on the solution of the integral equation, when the service times and the vacation times are both exponentially distributed.


\section{Model description and assumption} 

We consider a first-in-first-out single-server queueing system with single vacations in which customers are subject to a constant deadline $K > 0$ on the waiting time. A customer cannot wait more than $K$ time units in the queue. If he does not reach the server before a time $K$, he leaves the system and never returns. When he reaches the server, he remains until service completion. Customers arrive at periodic epochs, $T_n := nT$, $n \in \mathbb{N}$ and require service duration $\sigma_n$, $n \in \mathbb{N}$. We assume that the server is free initially, and the first customer begins to be served on arrival. Customers may renege from the queue or balk upon arrival. A balking customer do not join the queue at all, and a customer who reneges joins the queue but leaves without being served. We examine the single service discipline, i.e., after each service completion the server goes on vacations (\textbf{justifier}). At the end of its vacation period, the server begins serving if a customer is present; otherwise it remains idle until a new customer arrives. For $n > 0$, denote by $v_n$ a real-valued random variable representing the length of the $n$th vacation period. Both sequences $(v_n, n \geq 1)$ and $(\sigma_n , n \geq 0)$ are assumed to be independent and identically distributed non-negative random variables with distribution function $V(x)$ and $B(x)$. 
All random variables are defined on the same probability space $(\Omega , \mathcal{F}, \mathbb{P})$. 

In practice observable data may consist only in partial information. Suppose that, we observe $ \{(T_n, \sigma_n), n \geq 0\} $ and that the vacation periods are not observable. In this setting, the workload process which depends on $\{ \sigma_n, n \geq 0 \}$, $\{ v_n, n \geq 0 \}$ cannot be anticipated, so that the balking model is not appropriate. It is then desirable to investigate the reneging behaviour of customers.
In \cite{WardGlynn2005}, authors have taken into consideration the nature of the observable data by supposing that queue length is observable and have developed performance measure approximations for both reneging and balking models. As observed by Baccelli et al. \cite{Baccelli84}, customers who renege from the queue do not influence the waiting time of served customers. Thus, many steady-state performance measures are identical for reneging and balking models. \vspace{0.2cm}

In what follows, Lindley-type recursive equations are formulated for both reneging and balking models under the assumption that we observe the arrivals and the service duration of customers but the vacation periods are not directly observable. This implies that the periodic arrival point process will be marked only by the sequence $\{ \sigma_n, n \geq 0 \}$. 


\subsection{The waiting time process for the reneging model}

In the reneging model, all customers enter into the system. If the waiting time of a customer upon his arrival exceeds his patience time $K$, he abandons the queue without being served. Let $w_n$ be the waiting time experienced by the $n$th customer, i.e., the time between his arrival and the instant where he starts his service (if he receives a service). The idea is to express the waiting time of the $(n+1)$th customer in terms of those of the last served customer. \vspace{0.3cm}

$\bullet$ If the $n$th customer joins the server ($ w_n < K $), at his departure time the server takes one vacation of length $v_{N(n+1)}$, where $N(n+1)$ is the number of served customers prior to the $(n+1)$th arrival. Since a customer cannot wait more than $K$ time units, the waiting time of the $(n+1)$th customer is given by
\begin{eqnarray*}
w_{n+1} = \min[K, (w_n + \sigma_n + v_{N(n+1)} - T)^+],
\end{eqnarray*}
where
\begin{eqnarray}
\label{equation0}
N(n+1) = \sum_{k=0}^n \mathbf{1}_{(w_k < K)}
\end{eqnarray}
is the number of successfully served customers prior to the $(n+1)$th arrival. \vspace{0.5cm}

$\bullet$ If the $n$th customer abandons the queue without being served ($w_n = K$) and the $(n-1)$th customer joins the server ($w_{n-1} < K $), the waiting time $w_{n+1}$ satisfies
\begin{eqnarray*}
w_{n+1} = \min[K, (w_{n-1} + \sigma_{n-1} + v_{N(n)} - 2T)^+].
\end{eqnarray*}
The fact that the $n$th customer leaves the system without being served is expressed by
\begin{eqnarray*}
w_{n-1} + \sigma_{n-1} + v_{N(n)} - T = K.\vspace{0.5cm}
\end{eqnarray*}

$\bullet$ More generally, if $w_n = \ldots = w_{n-k+1} = K$ and $w_{n-k} < K$ for some $k = 0,1,\ldots, n$ we have
\begin{eqnarray} 
\label{equation1}
\left\{\begin{array}{lll}
                         w_0 = 0\\
                         w_{n+1} = \min [ (w_{n-k} + \sigma_{n-k} + v_{N(n-k+1)} - (k+1)T)^+,K ],\;\; n \geq 0,\\
                       \end{array}\right.
\end{eqnarray}
where $k$ is the number of lost customers between the $(n-k)$th and the $(n+1)$th customer. Furthermore, we have the following inequalities for each lost customer
\begin{equation}
\label{equation2}
w_{n-k} + \sigma_{n-k} + v_{N(n-k+1)} - (j+1)T = K, \quad j = 0, \ldots, k-1.
\end{equation}

\noindent
Equations \eqref{equation1} and \eqref{equation2} are particular cases of Equations (2) and (3) in \cite{Daley65}, where the author studied the general queueing system $GI/G/1 + GI$.

\subsection{The workload process for the balking model}

Let $\{ \tilde{w}_t, t\in\mathbb{R} \}$ be the workload process. The random variable $w_t$ represents the amount of work remaining to be done by the server at time $t$. By convention $\{ \tilde{w}_t, t\in\mathbb{R} \}$ will be taken
right-continuous with left limit $\tilde{w}_{t^-}$ and $\tilde{w}_{0^-} = 0$. We define the workload sequence
by, $\tilde{w}_n = \tilde{w}_{T_n^-} $, for all $n \in  \mathbb{N}$. Thus, the value
of $\tilde{w}_n$ taken up to time $T_n$ represents the time that the $n$th customer would have to wait to reach the server. The workload upon arrival of a customer is assumed to be known, hence a customer enters the system if and only if the workload upon his arrival is lower than his patience time $K$. If not, the customer does not enter and never returns. The server takes one vacation as soon as a customer completes his service. Consequently, vacation lengths are indexed by (\ref{equation0}) (by replacing $w_n$ by $\tilde{w}_n$). The general case was studied by \cite{Baccelli84}, or \cite{Stanford79}. We have for $n \geq 0$
\begin{eqnarray} \label{equation3}\left\{\begin{array}{lll}
\tilde{w}_0 = 0\\
\tilde{w}_{n+1} =  \left[\tilde{w}_n + (\sigma_n + v_{N(n+1)})\mathbf{1}_{(\tilde{w}_n < K)} -T \right]^+ .
\end{array}\right. \end{eqnarray}

\noindent
The above equation is similar to Equations (2.1) in \cite{Baccelli84}.

\section{Preliminary results}

In this section, we shall derive time-dependent integral equations for both the waiting time and the workload process. Let us, first introduce one lemma. \vspace{0.3cm}

Let $n$ and $p$ be two non-negative integers. Recall that $N(n+1) = \sum_{k=0}^n \mathbf{1}_{(w_k < K)}$ is the number of successfully served customers prior to $(n+1)T$ (for the reneging scenario). For $2 \leq p \leq n+1$, if $N(n+1) = p$ and  $w_n < K$, then the $n$th customer is the $p$th served customer. After his service, the server takes its $p$th vacation period of length $v_p$. Thus, the event $\{w_n < K, N(n+1) = p \}$ is a function of $\sigma_0, \ldots , \sigma_{n-1}, v_1,\ldots, v_{p-1}$, and $v_p$ is independent of the event $\{w_n < K, N(n+1) = p \}$. If $p > n+1$ or $p = 1$, then $\{N(n+1) = p, w_n < K\} = \emptyset $. Let $\sigma(\sigma_0, \ldots , \sigma_{n-1},v_1, \ldots, v_p)$ the $\sigma$-field generated by the random variables $\sigma_0, \ldots , \sigma_{n-1}, v_1, \ldots, v_p$. We have the following lemma.

\begin{lem} 
\label{lemma1}
Let $n \geq 1 $ and $p \geq 1$ be two non-negative integers. The events $\{w_n < K, N(n+1) = p\}$ and $\{\tilde{w}_n < K, N(n+1) = p\}$ are both $\sigma(\sigma_0, \ldots , \sigma_{n-1},v_1, \ldots, v_{p-1})$ measurable, and $v_p$ is independent of both events $\{w_n < K, N(n+1) = p\}$ and $\{\tilde{w}_n < K, N(n+1) = p\}$.
\end{lem}

Let $W_n(x)$ and $\tilde{W}_n(x)$, $x \in \mathbb{R}^+$ be the distribution functions of $w_n$ and $\tilde{w}_n$ respectively. For the reneging scenario, no customer can wait in the queue more than $K$ units of times. For all $n \geq 0$, $w_n$ is lower than $K$ with probability one. Thus, for $0 \leq x < K$ we have
\begin{equation*}
\mathbb{P}( w_{n+1} \leq x) = \sum_{k=0}^n \mathbb{P}( w_{n+1} \leq x , w_{n-k} < K, w_{n-j} = K, j = 0, \ldots k-1 ).
\end{equation*}
From Equations \eqref{equation1} and \eqref{equation2}, by conditioning first with respect to $w_n$, secondly, with respect to $N(n+1)$, and using Lemma~\ref{lemma1}, and the fact that both sequences $(\sigma_n, n \geq 0)$ and $(v_n, n \geq 1)$ are i.i.d. and mutually independent, yields for $0 \leq x < K$
\begin{align*}
\begin{split}
\mathbb{P}(w_{n+1} &\leq x , w_{n-k} < K, w_{n-j} = K, j = 0, \ldots k-1)\\
\qquad &=\int_{0-}^{K-0} [G(a_k^x(w)) - G(b_k(w))]dW_{n-k}(w), \quad k = 0,\ldots, n, 
\end{split}
\end{align*}

\noindent
where 
\begin{equation}
\label{def suites a et b}
\left\lbrace \begin{array}{lll}
  a_k^x(w) = x- w +(k+1)T, \; k \geq 0 \vspace{0.3cm}\\
  b_0(w) = 0, \quad b_k(w) = K-w+kT, \; k \geq 1,
\end{array} \right.
\end{equation}

\noindent
and $G$ denotes the distribution function of $\sigma_0 + v_1$. The previous equation is given with the condition that $G(s) - G(u) = 0$ for $s-u \leq 0$. For $0 \leq x < K$, the time-dependent integral equation for the waiting time for the reneging behaviour satisfies, for $ 0 \leq x < K$
\begin{equation}
\label{equation5}
W_{n+1}(x) = \sum_{k=0}^n \int_{0^-}^{K-0} \left\lbrace G(a_k^x(w)) - G(b_k(w))\right\rbrace dW_{n-k}(w).
\end{equation}
If $x\geq K$, $W_{n+1}(x) = 1$. \vspace{0.5cm}

\noindent
Simple calculations yield for the balking scenario
\begin{equation}
\label{equation6}
\tilde{W}_{n+1}(x) = \int_{0-}^{K-0} G(x - w  + T)d\tilde{W}_n(w) + \int_{K-0}^{T+x}d\tilde{W}_n(w).
\end{equation}
with the condition that $G(u) = 0$ and $dW(u) = 0$ for $u < 0$.

\section{Stability condition}

For all $n \geq 1$ let $u_n = \mathbb{P}(w_n = 0)$ be the probability that the $n$th customer finds the server free (with $u_0 = 1$), and $f_n = \mathbb{P}(w_n = 0,w_{n-1} > 0, \ldots, w_1 > 0)$, the probability that the event $\{ w_n = 0 \}$ occurs after $n$ steps (with $f_0 = 0$). For $k \geq 0$, let $\nu_0^k$ (with $\nu^{0}_{0} = 0$) be the number of entering customers in the $k$th busy period, that is, the duration for which the server is either serving or on vacations. We denote by $\mu$ the mean of the renewal epochs. 
The following result is a corollary of Theorem 1 in \cite{Daley65}.

\begin{cor} 
\label{thm1}
For the $D/GI/1$ queue with single vacations, constant deadline and single service discipline, assuming that $\mathbb{P}(\sigma_0 + v_1 - T < 0) > 0$, the limiting waiting time distribution function $W$ exists. Moreover, if the distribution functions $B$ and $V$ are continuous, then $W$ satisfies for $0 \leq x < K$
\begin{equation}
\label{equation7}
W(x) = \int_{0^-}^{K-0} \sum_{n \geq 0}G_n^x(w) dW(w),
\end{equation}

where
\begin{equation}
\label{equation8}
\sum_{n \geq 0}G_n^x(w) =  \sum_{n \geq 0} \mathbb{P}(b_n(w) \leq \sigma_0 + v_1 \leq a_n^x(w)), 
\end{equation}

\noindent
where $\{ a_n^{x} \}_{n \geq 0}$ and $\{ b_n\}_{n \geq 0}$ defined by Equation~\eqref{def suites a et b}. \vspace{0.3cm}

\noindent
The probability of rejection is given by
\begin{equation}
\label{equation9}
B_K := \lim_{n \rightarrow \infty}\mathbb{P}(w_n = K) = \int_{0^-}^{K-0} \sum_{n=1}^{\infty} \left[ 1 - G(K - w + nT) \right]dW(w).
\end{equation}
\end{cor}

\begin{rem}
\label{rem1}
Since the sequence $(v_n, n \geq 1)$ is i.i.d., the sequence $(w_n, n\geq 0 )$ has the same law than the sequence $ (z_n, n\geq 0)$, where $z_{n+1}$ is defined by $z_0 = 0$, $z_{n+1} = \min [ (z_{n-k} + \sigma_{n-k} + v_{n-k+1} - (k+1)T)^+,K ]$ for $n \geq 0$. It means, that the model that we propose in this paper, coincides (in law) with the classical $D/GI/1 +D$ queuing model for the reneging scenario, in which the sequence of service durations $(s_n, n \geq 0)$ is defined by $s_n := \sigma_n + v_{n+1}$, $n \geq 0$. In other words, the non observable data, $(v_n, n \geq 1)$ may be regarded as a sequence of marks for the arrival process.
\end{rem}

\begin{rem}
\label{rem2}
As in Remark~\ref{rem1}, this model coincides (in law) with the model defined by $\tilde{z}_0 = 0$, $\tilde{z}_{n+1} = [\tilde{z}_n + (\sigma_n + v_{n+1})\mathbf{1}_{(\tilde{z}_n < K)} - T]^+$. Thus, it may be reduce to the $D/GI/1 + D$ queue for balking customers.
\end{rem}

\begin{rem}
\label{rem3}
The $D/G/1 + D$ queue was studied by Ghosal in \cite{Ghosal1963}. The author derived an integral equation for the stationary waiting time based on the model introduced by Finch in \cite{Finch1960} (see also \cite{Ghosal1970}, Chapter 1, equation (1.4)), where the case $w_n=K$ is not taken into account (see \cite{Finch1961}, for a correction). Thus, the integral equation derived in \cite{Ghosal1963} does not take into account the case where customers left prematurely the queue. 
\end{rem}

\begin{lem}
\label{lemma2}
Under the condition $\mathbb{P}(\sigma_0 + v_1 < T) > 0$,  we have $\mu < \infty $.
\end{lem}
\begin{proof} 
The sequence $ (\tilde{z}_n, n\geq 0)$ defined in Remark 2, is a regenerative process with respect to the renewal sequence $ (\tau_0^k, k\geq 0)$, where $\tau_0^k$ is the number of customers entering during the $k$th busy period (wit h $\tau_0^0 = 0$) and provided $\mathbb{P}(\tau_0^1 < \infty) = 1$. When comparing $w_n$ with $\tilde{w}_n$ and using Remark 2 we have $w_n \leq \tilde{w}_n \stackrel{\mathcal{L}}{=} z_n$ for all $n\geq 0$. Since $\mathbb{P}(v_1 + \sigma_0 < T) > 0$, the renewal sequence $(\tau_0^k, k \geq 0)$ is aperiodic. From Theorem 2.2 in~\cite{Asmussen} we have $\mu \leq \tilde{\mu} = \mathbb{E}(\tau_0^1) = (q_0)^{-1} < \infty$, where $q_0 := \lim_{n \rightarrow \infty}\mathbb{P}(\tilde{z}_n = 0) > 0$. 
\end{proof}

\noindent
We now are able to prove Corollary~\ref{thm1}.

\begin{proof} 
\textit{Existence}. The first part of the proof is similar to that of Theorem 1 in \cite{Daley65}. For $ x \geq 0$, we introduce the function
\begin{equation*}
F_n(x) = \mathbf{P}(w_n \leq x, w_k > 0, k = 1, \ldots, n-1).
\end{equation*}
Then  $F_n(0) = f_n$ and $F_n(\infty) = f_n + f_{n+1}+ \ldots$ .

\noindent
One computes
\begin{align*}
\begin{array}{llll}
W_n(x) &= \sum_{k=0}^{n-1} \mathbf{P}(w_k = 0, w_{k+1} > 0, \ldots , w_n \leq x)\\
\\&=\sum_{k = 1}^n u_{n-k} F_k(x). \vspace{0.5 cm}\\
\end{array}
\end{align*}
Since $\sum_{n \geq 1}f_n = 1$ and $(f_n ,n \geq 1)$ is aperiodic the Theorem 2.2 of \cite{Asmussen} yields 
\begin{equation*}
\lim_{n \rightarrow \infty} u_n = \mu^{-1}.
\end{equation*}
Furthermore, for all $x \in[0,K)$, $0 \leq F_n(x) \leq  F_n(\infty)$,we obtain
\begin{equation*}
\sum_{n = 1}^{\infty}F_n(x) \leq  \sum_{n = 1}^{\infty}F_n(\infty) = \sum_{n = 1}^{\infty} nf_n  = \mu < \infty.
\end{equation*}
The series uniformly converges over $x \in [0,K)$. It follows from Theorem 1 p. 318 of ~\cite{FellerI} that the sequence $(W_n(x))$ converges uniformly for $x \in [0,K)$ and therefore the limit function $W(x)$ is a distribution function.

\noindent \\
\textit{Limit value}. For any integer $m$, we introduce the subdivision selected on continuity points of $W$
\begin{eqnarray*}
0 = w_{m,0} < w_{m,1} < \ldots < w_{m,l_m} = K,
\end{eqnarray*}
and $\Delta_m = \sup_{1 \leq j \leq l_m}(w_{m,j} - w_{m,j-1})$ such that $\lim_{m \rightarrow \infty} \Delta_m = 0$. For $x$ and $w$ $\in [0,K)$, define the sequence $(G^x_k)_{k \geq 0}$  such that
\begin{eqnarray*}
\left\lbrace \begin{array}{lll}
                 G_0^x(w) = G(a^x_0(w)),\vspace{0.3cm}
                 G_k^x(w) = G(a_k^x(w)) - G(b_k(w)), \;\; k \geq 1.                                          
\end{array} \right. 
\end{eqnarray*}
The function $G_k^x$ depends on $T,K$ and $\sigma$. For sake of simplicity, we omitted these parameters. For all $k$, the functions $G_k^x(w)$ are continuous on $[0,K)$ uniformly over $x$. Thus, according to the definition of Riemann-Stieltjes integral

\begin{equation}
\label{equation10}
\lim_{n \rightarrow \infty} W_{n+1}(x) =\lim_{n \rightarrow \infty} \lim_{m \rightarrow \infty} \sum_{j=1}^{l_m}\sum_{k=0}^n G^x_k (w_{m,j-1})[W_{n-k}(w_{m,j}) - W_{n-k}(w_{m,j-1})].
\end{equation}

\noindent
Define the sequence $(\alpha_{m,n})_{m \geq 0, n \geq 0}$ by
\begin{equation*}
\alpha_{m,n} = \sum_{j=1}^{l_m}\sum_{k=0}^n G^x_k(w_{m,j-1})[W_{n-k}(w_{m,j}) - W_{n-k}(w_{m,j-1})].
\end{equation*}

\noindent
The sequence $W_n$ and the series $ \sum_{n=0}^{\infty} G^x_n(w) $  are uniformly convergent. Thus, according to Theorem 1 p. 318 in \cite{FellerI}, the sequence $(\alpha_{m,n})_{n \geq 0}$ converges uniformly over $m$ to
\begin{eqnarray*}
\sum_{j=1}^{l_m}\sum_{n=0}^{\infty} G^x_n (w_{m,j-1})[W(w_{m,j}) - W(w_{m,j-1})].
\end{eqnarray*}

\noindent
Furthermore, the definition of Riemann-Stieltjes integral gives the convergence of $(\alpha_{m,n})_{m \geq 0}$ to 
\[ \int_{0^-}^{K-0} \sum_{k=0}^n G_k^x(w)dW_{n-k}(w).\]

\noindent
Thus, we may invert the limit in (\ref{equation10}) which yields (\ref{equation7}).

\noindent \\
\textit{Probability of rejection.} This probability is expressed as
\begin{align*}
\mathbb{P}(w_n = K) &= \sum_{k=1}^{n} \mathbb{P}(w_n = K, w_{n-k} < K, w_{n-j} = K, j = 1, \ldots , k-1)\\
                    & = \sum_{k=1}^{n} \int_{0^-}^{K-0} \left[ 1 - G(b_k(w)) \right]dW_{n-k}.
\end{align*}

\noindent
Since $W(x)$ exists, (\ref{equation9}) is proved.
 
\end{proof}

\noindent
Corollary~\ref{thm2} is similar to Corollary~\ref{thm1}. It establishes the sufficient condition for the stability for the balking scenario and gives the stationary workload distribution together with the blocking probability
$B_K := \underset{n \rightarrow \infty}{\lim}  \mathbb{P}(\tilde{w}_n \geq K)$.

\begin{cor} 
\label{thm2}
For the $D/GI/1$ queue with single vacations, constant deadline and single service discipline, assuming that $\mathbb{P}(v_1 + \sigma_0  < T) > 0  $, the limiting waiting time distribution function $\tilde{W}$ exists. Moreover, if the distribution function $G$ is continuous on $\mathbb{R}^+$, then $\tilde{W}$ satisfies
\begin{equation}
\label{equation01}
\tilde{W}(x) = \int_{0-}^{K-0} G(x - w + T)d\tilde{W}(w) + \int_{K-0}^{x+T} d\tilde{W}(w), \quad x \geq 0.
\end{equation}

\noindent
The blocking probability is given by
\begin{equation*}
B_K = \int_{0^-}^{K-0} \left[ 1 - G(K - w + T) \right]d\tilde{W}(w) + \int_{K+T}^{\infty}d\tilde{W}(w).
\end{equation*}
\end{cor}

\begin{proof}
The first part of the proof is similar to that of Corollary~\ref{thm1}. The limit $\tilde{W}$ is obtained by an Helly-Bray type argument (see, for example, ~\cite{LoeveI}). 
\end{proof}

\begin{rem}
\label{rem4}
Integrating by parts the first term of Equation~\eqref{equation01}) leads to
\[  G(x+T-K) - \int_{0^-}^{K-0} W(w)dG(T+x-w), \]
with the condition that $G(u) = 0$ and $G(u) = 0$ for $u < 0$. This equation is the same as Equation (1) of Ghosal in \cite{Ghosal1963}, but the author does not consider the case where the previous customer abandons the queue without being served.
\end{rem}

\section{The case of deterministic service and exponentially distributed vacation times}

In this section, we analyse the reneging model under the assumptions that customers require a deterministic service duration $\sigma > 0$ and the vacations are exponentially distributed with parameter $\lambda > 0$.  Equation \eqref{equation7} becomes
\begin{equation}
\label{equation19}
W(x) = \int_{0^-}^{K-0} \sum_{n \geq 0} V_n^x(w) dW(w),
\end{equation}

\noindent
where
\begin{equation*}
\sum_{n \geq 0} V_n^x(w) =  \sum_{n \geq 0} \mathbb{P}(b_n(w) \leq  v \leq a_n^x(w)),
\end{equation*}

\noindent 
and 
\begin{eqnarray*}
\left\{\begin{array}{lll} 
a_n^{x,T}(w) := x - \sigma - w + (n+1)T, \hspace{1.4 cm} n \geq 0, \vspace{0.3 cm}\\
b_0 := 0, \quad b_n^{K,T}(w) := K - \sigma - w + kT, \hspace{0.6 cm}  n \geq 1.\\
\end{array}\right. \vspace{0.3 cm}
\end{eqnarray*}

\noindent
Substituting $V(x) = 1 - e^{-\lambda x}$, $x \geq 0$ in \eqref{equation19} gives, for $0 \leq x < K, \; 0 \leq w < K$
\begin{equation}
\label{equation20}
\begin{split}
W(x) &=  1 - e^{-\lambda(x-\sigma +T)} + \alpha_{\lambda}[e^{-\lambda(K-\sigma)} - e^{-\lambda(x-\sigma+T)}] \\
& \quad + \int_{0^+}^{K-0} 1 - e^{-\lambda (x - w - \sigma + T)} dW(w) + \alpha_{\lambda} \int_{0^+}^{K-0} e^{-\lambda(K - w - \sigma)} - e^{-\lambda (x - w - \sigma + T)} dW(w),
\end{split} 
\end{equation}
where $\alpha_{\lambda} = \dfrac{e^{-\lambda T}}{1 - e^{-\lambda T}}$. From Equation~\eqref{equation20} and the Lebesgue's dominated convergence theorem it follows that
\begin{equation*}
\lim_{n \rightarrow \infty} \dfrac{W(x+h_n)-W(x)}{h_n} \leq \alpha_{\lambda} \lambda e^{\lambda(K+\sigma)},
\end{equation*}
where $ h_n \longrightarrow 0$ as $n \rightarrow \infty$. 
The distribution function $W$ is differentiable on $(0,K)$ and has a bounded derivative with a finite number of discontinuity (at points $ x=0$ and $x=K$). Hence $W$ is absolutely continuous with respect to the Lebesgue measure and we denote by $f$ the probability density function (pdf) of $W$. Taking the derivative in Equation~\eqref{equation14} yields
\begin{equation}
\label{equation21}
f(x) = W(0)\alpha e^{\lambda \sigma}\lambda e^{-\lambda x}  +  \lambda \alpha e^{\lambda \sigma}\int_{0}^x  e^{-\lambda(x - w)}f(w)dw, \quad  0 < x < K. 
\end{equation}

\noindent
Suppose that there exists a function $g$ satisfying, for all $x > 0$,
\begin{equation}
\label{equation22}
g(x) = G(0)\alpha e^{\lambda \sigma}\lambda e^{-\lambda x}  +  \lambda \alpha e^{\lambda \sigma}\int_{0}^x  e^{-\lambda(x - w)}g(w)dw,
\end{equation}

\noindent
where $G$ is the distribution function of $g$, and $W(0) = G(0)$. Taking the Laplace transform of (\ref{equation22}), we have
\begin{eqnarray*}
\Phi(\theta) =  W(0)\alpha e^{\lambda \sigma}\dfrac{\lambda}{\lambda + \theta} + \alpha e^{\lambda \sigma}\dfrac{\lambda}{\lambda + \theta} \Phi(\theta),\quad \theta > \lambda(\alpha e^{\lambda \sigma} - 1).
\end{eqnarray*}

\noindent
Rewriting this last equation gives
\begin{eqnarray*} 
\Phi(\theta) = \dfrac{G(0)\lambda \alpha e^{\lambda \sigma} }{\theta - \lambda(\alpha e^{\lambda \sigma} -1)}.
\end{eqnarray*}

\noindent
Hence, by inversion we obtain 
\begin{equation*}
g(x) = G(0)\lambda\alpha e^{\lambda \sigma} e^{\lambda (\alpha  e^{\lambda \sigma} -1 )x}, \quad x > 0.
\end{equation*}

\noindent
Identifying $f$ with $g$ on $(0,K)$ leads to
\begin{equation}
\label{equation23}
f(x) = W(0)\lambda\alpha e^{\lambda \sigma} e^{\lambda (\alpha  e^{\lambda \sigma} -1 )x}, \quad  0 < x < K.
\end{equation}

\noindent
The constant $W(0)$ is evaluated by the condition
\begin{eqnarray*}
W(0) + \int_0^K f(x)dx + B_K = 1
\end{eqnarray*}

\noindent
thus
\begin{align}
\label{equation24}
W(0) &= [1 - B_K] \left[ 1 + \int_{0^-}^K \lambda\alpha e^{\lambda \sigma} e^{\lambda (\alpha  e^{\lambda \sigma} -1 )x} dx \right]^{-1}  \nonumber \\
 &= [1 -B_K] \left[ \dfrac{\alpha_{\lambda} e^{\lambda \sigma} - 1}{\alpha_{\lambda} e^{\lambda(K \alpha_{\lambda}e^{\lambda \sigma} - K + \sigma)} -1} \right],
\end{align}

\noindent 
where $B_K$ is the probability of rejection
\begin{equation}
\label{equation25}
B_K = \int_{0^-}^{K-0} \sum_{n=1}^{\infty} \left[ 1 - V(K - w - \sigma + nT) \right]dW(w).
\end{equation}

\noindent
Substituting $V(x) = 1 - e^{-\lambda x}$ in (\ref{equation25}) gives
\begin{equation}
\label{equation26}
B_K = W(0)\alpha_{\lambda} e^{-\lambda(K-\sigma-\alpha_{\lambda}e^{\lambda \sigma})}.
\end{equation}

\noindent
We have the following Proposition.

\begin{prp}
In the $D/D/1$ queue with exponential vacation times and single service discipline, under the conditions of Theorem 1, the density of $W$ on $(0,K)$ is
\begin{equation*}
f(x) = W(0)\lambda\alpha e^{\lambda \sigma} e^{\lambda (\alpha  e^{\lambda \sigma} -1 )x},
\end{equation*}

\noindent
where $W(0)$ and $B_K$ are given by (\ref{equation24}), (\ref{equation26}) respectively.
\end{prp}

\begin{rem}
\label{rem5}
Equation (\ref{equation21}) is of the type 
\begin{equation*}
f(x) = h(x) + \Lambda \int_{0}^x K(x-w)f(w)dw, \quad  0 < x < K,
\end{equation*}

\noindent
which is the so-called Volterra equation. Applying the method of the resolvent, we obtain
\begin{equation*}
K_{n+1}(x,w) = \dfrac{(x-w)^n}{n!}e^{-\lambda(x-w)}, \quad n \geq 0.
\end{equation*}

\noindent
Therefore, the resolvent kernel of (\ref{equation21}) is
\begin{equation*}
R(x,w;\Lambda) = e^{(\Lambda -\lambda)(x-w)},
\end{equation*}

\noindent
and the solution is given by
\begin{equation*}
f(x) = h(x) + \Lambda \int_{0}^x e^{(\Lambda -\lambda)(x-w)} h(w)dw.
\end{equation*}

\noindent
The calculation of the previous equation yields (\ref{equation23}).
\end{rem}

We now focus on the time-dependent waiting time distribution. 

\begin{prp}
For all $n \geq 0$ and $x \geq 0$ the time-dependent distribution of the waiting time for the balking model is given by
\begin{equation*}
\mathbb{P}(\tilde{w}_{n+1} > x) = \sum_{j=0}^n \binom{n}{j}K^{n-j}\lambda^{n-j}e^{-\lambda(x+(n+1)T - (n+1-j)\sigma)}.
\end{equation*}
\end{prp}

\begin{proof}
The proof is by induction. For $n = 0$, we have $\mathbb{P}(\tilde{w}_1 > x) = \mathbb{P}(\sigma + v_1 - T > x) = e^{-\lambda(x + T - \sigma)}$. Assume that $\mathbb{P}(\tilde{w}_{n} > x) = \sum_{j=0}^{n-1} \binom{n-1}{j}K^{n-1-j}\lambda^{n-1-j}e^{-\lambda(x+nT - (n-j)\sigma)}$ hold. Then,

\begin{equation*}
\begin{split}
\mathbb{P}(\tilde{w}_{n+1} > x, \tilde{w}_n < K) &= \int_{0}^K \mathbb{P}(w + \sigma + v_{n+1} - T > x) \\
  & \qquad \times\lambda \sum_{j=0}^{n-1} \binom{n-1}{j}K^{n-1-j}\lambda^{n-1-j}e^{-\lambda(x+nT - (n-j)\sigma)}\\
                    &= \sum_{j=0}^{n-1} \binom{n-1}{j}K^{n-j}\lambda^{n-j}e^{-\lambda(x + (n+1)T - (n+1-j)\sigma)}.
\end{split}
\end{equation*}

\noindent
consequently, $\mathbb{P}(\tilde{w}_{n+1} > x, \tilde{w}_n  \geq K) = \sum_{j=0}^{n-1} \binom{n-1}{j}K^{n-1-j}\lambda^{n-1-j}e^{-\lambda(x+(n+1)T - (n-j)\sigma)}$. Hence,

\begin{equation*}
\begin{split}
\mathbb{P}(\tilde{w}_{n+1} > x) &= \sum_{j=0}^{n-1} \binom{n-1}{j}K^{n-j}\lambda^{n-j}e^{-\lambda(x + (n+1)T - (n+1-j)\sigma)} \\
                                & \qquad + \sum_{j=0}^{n-1} \binom{n-1}{j}K^{n-1-j}\lambda^{n-1-j}e^{-\lambda(x+(n+1)T - (n-j)\sigma)} \\
                                & =\binom{n-1}{0}K^n \lambda^n + e^{-\lambda(x+(n+1)T - (n+1)\sigma)} \\
                         & \qquad + \sum_{j=1}^{n-1} \binom{n-1}{j}K^{n-j}\lambda^{n-j}e^{-\lambda(x + (n+1)T - (n+1-j)\sigma)}\\
                                & \qquad + \sum_{j=0}^{n-2} \binom{n-1}{j}K^{n-1-j}\lambda^{n-1-j}e^{-\lambda(x+(n+1)T - (n-j)\sigma)}\\
                              & \qquad + \binom{n-1}{n-1}e^{-\lambda (x+(n+1)T - \sigma)}\\
                                &=\sum_{j=0}^n \binom{n}{j}K^{n-j}\lambda^{n-j}e^{-\lambda(x+(n+1)T - (n+1-j)\sigma)}.
\end{split}
\end{equation*}

\begin{equation*}
\begin{split}
\mathbb{P}(\tilde{w}_{n+1} > x) &= \sum_{j=0}^{n-1} \binom{n-1}{j}K^{n-j}\lambda^{n-j}e^{-\lambda(x + (n+1)T - (n+1-j)\sigma)} \\
                                & \qquad + \sum_{j=0}^{n-1} \binom{n-1}{j}K^{n-1-j}\lambda^{n-1-j}e^{-\lambda(x+(n+1)T - (n-j)\sigma)} \\
                                &=\sum_{j=0}^n \binom{n}{j}K^{n-j}\lambda^{n-j}e^{-\lambda(x+(n+1)T - (n+1-j)\sigma)}.
\end{split}
\end{equation*}

\end{proof}

\section{The case of exponentially distributed service times and vacation times}

In this section, from \eqref{equation7} we obtain differential equation for the unknown pdf $f$ and solve it explicitly using the Laplace transform as in the previous section. Throughout this section, we assume that $B(x) = 1 - e^{-\mu x}$ and $V(x) = 1 - e^{-\lambda x}$, $x \geq 0$ with $\lambda >0$, $\mu > 0$ and $\lambda \neq \mu$. Let $G(x)$ be the distribution function of the random variable $\sigma_0 + v_1$, that is for $ x \geq 0$
\begin{equation*}
G(x) = 1 - \mu/(\mu - \lambda)e^{-\lambda x} -  \lambda/(\lambda -  \mu)e^{-\mu x}.
\end{equation*}

\noindent
Equation \eqref{equation8} becomes 
\begin{equation*}
\begin{split}
\sum_{n \geq 0} G^x_n(w) &= \mathbb{P}(\sigma + v \leq x -w +T)\\
& \quad +\sum_{n \geq 1}\mathbb{P}(K -w + nT \leq \sigma + v \leq x -w +(n+1)T) \\
   &= 1 -  \dfrac{\mu}{\mu - \lambda}e^{-\lambda(x - w + T)} - \dfrac{\lambda}{\lambda -  \mu}e^{-\mu (x - w + T)} + \dfrac{\mu}{\mu - \lambda} \alpha_{\lambda}  \left[e^{-\lambda(K - w)} - e^{-\lambda(x - w + T)}  \right]  \\
     & \qquad + \dfrac{\lambda}{\lambda -  \mu} \alpha_{\mu} \left[e^{-\mu(K - w)} - e^{-\mu(x - w + T)}\right] ,
\end{split}   
\end{equation*}

\noindent
where
\[ \alpha_{\lambda} = e^{-\lambda T}/(1 - e^{-\lambda T})  \qquad \mbox{and} \qquad \alpha_{\mu} = e^{-\mu T}/(1 - e^{-\mu T} ). \]

\noindent
We now derive the stationary waiting time integral equation. Substituting the above equation in \eqref{equation7}q, yields for $0 \leq x < K$,
\begin{equation}
\label{equation11}
\begin{split}
 W(x) &= W(0)\sum_{n \geq 0} G^x_n(0) + \int_{0^+}^{K-0} \sum_{n \geq 0} G^x_n(w)dW(w)\\
      &= W(0) \left\{ 1 - \dfrac{\mu}{\mu - \lambda} e^{-\lambda(x+T)} - \dfrac{\lambda}{\lambda - \mu}e^{-\mu(x+T)} 
       + \dfrac{\mu}{\mu - \lambda} \alpha_{\lambda} \left( e^{-\lambda K} - e^{-\lambda (x+T)} \right) \right. \\
      & \qquad \qquad \; \left. +  \dfrac{\lambda}{\lambda - \mu} \alpha_{\mu} \left( e^{-\mu K} - e^{-\mu (x+T)} \right) \right\} \\ 
      & \qquad + \int_{0^+}^{K-0}\left[ 1 - \dfrac{\mu}{\mu - \lambda} e^{-\lambda(x - w + T)} - \dfrac{\lambda}{\lambda - \mu}e^{-\mu(x - w + T)} \right] dW(w) \\
      & \qquad + \dfrac{\mu}{\mu - \lambda} \alpha_{\lambda} \int_{0^+}^{K-0}\left[ e^{-\lambda(K-w)} - e^{-\lambda (x-w+T)} \right]dW(w)\\
      & \qquad + \dfrac{\lambda}{\lambda - \mu} \alpha_{\mu}\int_{0^+}^{K-0}\left[ e^{-\mu (K-w)} - e^{-\mu (x-w+T)} \right] dW(w). 
\end{split}      
\end{equation}
 
\noindent
Taking the derivative of \eqref{equation11} with respect to $x$ yields
\begin{equation}
\label{equation12}
\begin{split}
   f(x)   &= W(0) \left\lbrace \dfrac{\lambda \mu}{\mu - \lambda} \alpha_{\lambda} e^{-\lambda x}  + \dfrac{\lambda \mu}{\lambda - \mu} \alpha_{\mu} e^{-\mu x}\right\rbrace \\
     &\quad + \quad \dfrac{\lambda \mu}{\mu - \lambda} \alpha_{\lambda} \int_{0}^{x}  e^{-\lambda(x - w)} f(w)dw  + \dfrac{\lambda \mu}{\lambda - \mu} \alpha_{\mu}\int_{0}^{x}  e^{-\mu(x - w)} f(w)dw.
\end{split}      
\end{equation}

\noindent
Hereinafter, we transform \eqref{equation12}q into an $2$-nd order linear homogeneous differential equation. Taking the derivative with respect to $x$ to Equation~\eqref{equation12} yields
\begin{equation}
\label{equation13}
\begin{split}
f'(x) &= -\lambda \left\lbrace W(0) \dfrac{\lambda \mu}{\mu - \lambda}\alpha_{\lambda}e^{-\lambda x} + \dfrac{\lambda \mu}{\mu - \lambda} \alpha_{\lambda} \int_{0}^x e^{-\lambda (x-w)}f(w)dw \right\rbrace \\
  & \quad  -\mu \left\lbrace W(0) \dfrac{\lambda \mu}{\lambda - \mu}\alpha_{\mu}e^{-\mu x} + \dfrac{\lambda \mu}{\lambda - \mu} \alpha_{\mu} \int_{0}^x e^{-\mu (x-w)}f(w)dw \right\rbrace\\
    & \quad + \dfrac{\lambda \mu}{\lambda - \mu}\alpha_{\mu} f(x)+ \dfrac{\lambda \mu}{\mu - \lambda}\alpha_{\lambda} f(x).\\
\end{split}
\end{equation}

\noindent
Equation (\ref{equation13}) can be rewritten as
\begin{equation*}
\begin{split}
f'(x) &=  \left( -\lambda - \mu + \dfrac{\lambda \mu}{\mu - \lambda}\alpha_{\lambda} + \dfrac{\lambda \mu}{\lambda - \mu}\alpha_{\mu} \right)f(x)\\
  & \qquad+ \lambda \left\lbrace W(0) \dfrac{\lambda \mu}{\lambda - \mu}\alpha_{\mu}e^{-\mu x} + \dfrac{\lambda \mu}{\lambda - \mu} \alpha_{\mu} \int_{0}^x e^{-\mu (x-w)}f(w)dw \right\rbrace\\
      & \qquad + \mu \left\lbrace W(0) \dfrac{\lambda \mu}{\mu - \lambda}\alpha_{\lambda}e^{-\lambda x} + \dfrac{\lambda \mu}{\mu - \lambda} \alpha_{\lambda} \int_{0}^x e^{-\lambda (x-w)}f(w)dw \right\rbrace.
\end{split}
\end{equation*}

\noindent
Differentiating the above equation leads to the following $2$-nd order homogeneous linear differential equation
\begin{equation}
\label{equation14}
f''(x) + A_{\lambda,\mu}f'(x) + B_{\lambda,\mu}f(x) = 0,
\end{equation}

\noindent
where
\[ A_{\lambda,\mu} = \lambda + \mu - \dfrac{\lambda \mu}{\mu - \lambda}\alpha_{\lambda} - \dfrac{\lambda \mu}{\lambda - \mu}\alpha_{\mu} \qquad \mbox{and}\qquad B_{\lambda,\mu} = \lambda \mu - \dfrac{\lambda^2 \mu}{\lambda - \mu} \alpha_{\mu} - \dfrac{\lambda \mu^2}{\mu - \lambda}\alpha_{\lambda}. \]

\noindent
To solve (\ref{equation14}), we will apply the Laplace transform. The density $f$ has a bounded support. In order to inverse the Laplace transform, we shall introduce a function $g$ with support $(0, \infty)$. Supppose that there exists a function $g$ which coincides with $f$ on $(0,K)$, so that it satisfies for all $x > 0$,
\begin{equation*}
\begin{split}
   g(x)   &= G(0) \left\lbrace \dfrac{\lambda \mu}{\mu - \lambda} \alpha_{\lambda} e^{-\lambda x} + \dfrac{\lambda \mu}{\lambda - \mu} \alpha_{\mu} e^{-\mu x} \right\rbrace
             + \dfrac{\lambda \mu}{\mu - \lambda} \alpha_{\lambda} \int_{0}^{x}  e^{-\lambda(x - w)} g(w)dw \\
      & \qquad + \dfrac{\lambda \mu}{\lambda - \mu} \alpha_{\mu}\int_{0}^{x}  e^{-\mu(x - w)} g(w)dw,
\end{split}      
\end{equation*}

\noindent
where $G$ denotes the probability function of $g$, and
\begin{equation}
\label{equation15}
g''(x) + A_{\lambda,\mu}g'(x) + B_{\lambda,\mu}g(x) = 0.
\end{equation}

\noindent
Assume futhermore, that $W(0) = G(0)$. Let $\Phi$ be the Laplace transform of $g$, 
\begin{equation*}
\Phi(\theta) = \int_{0}^{\infty} e^{- \theta x} g(x) dx, \qquad x \in \mathbb{R}^+, \quad \theta \in \mathbb{C}, \quad \mbox{Re}(\theta) \geq 0.
\end{equation*}

\noindent
By taking the Laplace transform of (\ref{equation15}), we obtain
\begin{equation*}
(\theta^2 + A_{\lambda, \mu}\theta + B_{\lambda,\mu})\Phi(\theta) = (A_{\lambda,\mu} + \theta)g(0) + g'(0).
\end{equation*}

\noindent
Rearranging the terms gives
\begin{equation}
 \label{racine gamma}
\dfrac{\Phi(\theta)}{G(0)} = \dfrac{\theta \lambda \mu(\alpha_{\lambda} - \alpha_{\mu}) + \lambda \mu (\mu \alpha_{\lambda} - \lambda \alpha_{\mu}) }{\theta^2 (\mu - \lambda) + \theta \left[ (\mu ^2 - \lambda ^2) + \lambda \mu (\alpha_{\mu} - \alpha_{\lambda}) \right] + \lambda \mu  \left[ \mu (1 - \alpha_{\lambda}) - \lambda (1 - \alpha_{\mu})\right]}.
\end{equation}

\noindent
The denominator of $\dfrac{\Phi(\theta)}{G(0)}$ is clearly polynomial of degree two. Its roots will be denoted $\gamma_1$ and $\gamma_2$ and have negative real parts. Furthermore, the derivative has one zero which is not a root of the denominator, thus $\gamma_i$ are simple roots ($i=1,2$). The numerator is polynomial of degree one, and we have the following partial fraction expansion
\begin{equation}
\label{equation16}
\dfrac{\Phi(\theta)}{G(0)} = \sum_{i=1}^2 \dfrac{C_i}{\theta - \gamma_i}.
\end{equation}

\noindent
The constants $C_i$ ($i = 1,2$), are expressed by
\begin{eqnarray}
\label{constante C1}
C_1 = \lim_{\theta \rightarrow \gamma_1} \dfrac{\Phi(\theta)}{G(0)}(\theta - \gamma_1) = \dfrac{\gamma_1 \lambda \mu(\alpha_{\lambda} - \alpha_{\mu}) + \lambda \mu(\mu \alpha_{\lambda} - \lambda \alpha_{\mu}) }{\gamma_1 - \gamma_2},
\end{eqnarray}

\noindent
and
\begin{eqnarray}
\label{constante C2}
C_2 = \lim_{\theta \rightarrow \gamma_2} \dfrac{\Phi(\theta)}{G(0)}(\theta - \gamma_2) = \dfrac{\gamma_2 \lambda \mu(\alpha_{\lambda} - \alpha_{\mu}) + \lambda \mu(\mu \alpha_{\lambda} - \lambda \alpha_{\mu}) }{\gamma_2 - \gamma_1}.
\end{eqnarray}

\noindent
Inverting (\ref{equation16}) yields
\begin{eqnarray*}
g(x) = G(0)\sum_{i=1}^2 C_i e^{\gamma_i x}, \qquad x > 0 .
\end{eqnarray*}

\noindent
Identifying $f(x)$ with $g(x)$ on $0 < x < K$, gives
\begin{eqnarray*}
f(x) = W(0)\sum_{i=1}^2 C_i e^{\gamma_i x}.
\end{eqnarray*}

\noindent
It remains to find the constant $W(0)$ which is done by the normalizing condition
\begin{eqnarray*}
W(0) + \int_{0}^K f(x) dx + B_K = 1.
\end{eqnarray*}

\noindent
Therefore,
\begin{eqnarray}
\label{equation17}
W(0) = \left[ 1 - B_K \right]\left[  1 + \int_{0}^K \sum_{i=1}^2 C_i e^{\gamma_i x} dx \right]^{-1}.
\end{eqnarray}

\noindent
The calculation of the probability of rejection comes from to (\ref{equation9}), we have
\begin{equation*}
\begin{split}
B_K &= W(0) \left\lbrace \dfrac{\mu}{\mu - \lambda}\alpha_{\lambda}e^{-\lambda K} + \dfrac{\lambda}{\lambda - \mu}\alpha_{\mu} e^{-\mu K} \right\rbrace \\
    & \qquad + \dfrac{\mu}{\mu - \lambda}\alpha_{\lambda}\int_{0}^K e^{\lambda w} f(w) dw + \dfrac{\lambda}{\lambda - \mu}\alpha_{\mu} e^{-\mu K} \int_{0}^K e^{\mu w} f(w) dw.
\end{split}    
\end{equation*}

\noindent
The calculation of integrals in the last equality gives
\begin{eqnarray*}
\int_{0}^K e^{\lambda w} f(w) dw = \dfrac{W(0)C_1}{\lambda + \gamma_1}\left[ e^{(\lambda + \gamma_1)K} - 1 \right] + \dfrac{W(0)C_2}{\lambda + \gamma_2}\left[ e^{(\lambda + \gamma_2)K} - 1 \right],
\end{eqnarray*} 

\begin{eqnarray*}
\int_{0}^K e^{\mu w} f(w) dw = \dfrac{W(0)C_1}{\mu + \gamma_1}\left[ e^{(\mu + \gamma_1)K} - 1 \right] + \dfrac{W(0)C_2}{\mu + \gamma_2}\left[ e^{(\mu + \gamma_2)K} - 1 \right].
\end{eqnarray*} 

\noindent
Finally,
\begin{equation}
\label{equation18}
\begin{split}
B_K &= W(0)  \dfrac{\mu}{\mu - \lambda}\alpha_{\lambda} \left\lbrace e^{-\lambda K} + \dfrac{C_1}{\lambda + \gamma_1}\left[ e^{(\lambda + \gamma_1)K} - 1 \right] + \dfrac{C_2}{\lambda + \gamma_2}\left[ e^{(\lambda + \gamma_2)K} - 1 \right]\right\rbrace \\
    & + W(0)\dfrac{\lambda}{\lambda - \mu}\alpha_{\mu} \left\lbrace e^{-\mu K} + \dfrac{C_1}{\mu + \gamma_1}\left[ e^{(\mu + \gamma_1)K} - 1 \right] + \dfrac{C_2}{\mu + \gamma_2}\left[ e^{(\mu + \gamma_2)K} - 1 \right]\right\rbrace .
\end{split}           
\end{equation}

\begin{rem}
Equation (\ref{equation14}) may be solved using the characteristic equation $t^2 + A_{\lambda, \mu}t + B_{\lambda, \mu} = 0$. The solution of (\ref{equation14}) is of the form $f(x) = C_1 e^{t_1 x} + C_2 e^{t_2 x}$, where $t_1 \neq t_2$ are the roots of the characteristic equation and $C_1$, $C_2$ are calculated from initial or boundary conditions.
\end{rem}

\noindent
These conclusions are summarized in the following Proposition.

\begin{prp}
In the $D/M/1$ queue with exponential vacation times and single service discipline, under conditions of Proposition~\eqref{thm1}, the pdf of $W$ on $(0,K)$ is given by
\begin{equation*}
f(x) = W(0)\sum_{i=1}^2 C_i e^{\gamma_i x},
\end{equation*}

\noindent
where $W(0)$ and $B_K$ are given by \eqref{equation17} and \eqref{equation18} respectively, the constants $C_1$ and $C_{2}$ are given by \eqref{constante C1} and \eqref{constante C2}, and $\gamma_{i}$, $i=1,2$ are the roots of \eqref{racine gamma}. 

\end{prp}

\end{document}